\documentclass[10pt,a4paper]{amsart}
\usepackage{amsmath,amsfonts,amssymb,amsthm,mathdots}
\usepackage[utf8x]{inputenc}

\theoremstyle{plain}
\newtheorem{thmABC}{Theorem}

\newtheorem{corABC}[thmABC]{Corollary}
 \newtheorem{lem}{Lemma}
\newtheorem{claim}{Claim}

 \theoremstyle{definition}
 \newtheorem{defn}{Definition}
 \newtheorem{notn}{Notation}

\theoremstyle{remark}
\newtheorem*{rmk*}{Remark}

\newcommand{\vvert}{\ \vert\ }
\newcommand{\norm}[1]{{\vert #1 \vert}}
\newcommand{\wt}[1]{\widetilde{#1}}
\newcommand{\wh}[1]{\widehat{#1}}
\newcommand{\cwr}{\mbox{\textnormal{\small\textcircled{$\wr$}}}}
\newcommand{\comm}[1]{}
\newcommand{\mS}{\mathcal{S}}
\newcommand{\Sym}{\mathrm{Sym}}
\renewcommand{\a}{\alpha}

\title[On HJI groups with complete Hausdorff dimension spectrum]{On hereditarily just infinite profinite groups with complete Hausdorff dimension spectrum}
\author{YIFTACH BARNEA} \address{Department of Mathematics, Royal Holloway, University of London, Egham, Surrey TW20 0EX, UK}\email{y.barnea@rhul.ac.uk}

\author{MATTEO VANNACCI} \address{Mathematisches Institut der
  Heinrich-Heine-Universit\"at, Universit\"atsstr.\ 1, 40225
  D\"usseldorf, Germany}\email{matteo.vannacci@uni-duesseldorf.de}



\begin{document}

\maketitle

\begin{abstract}
 We prove that the inverse limit of certain iterated wreath products in product action have complete Hausdorff dimension spectrum with respect to their unique maximal filtration of open normal subgroups. Moreover we can produce explicitly subgroups with a specified Hausdorff dimension.
\end{abstract}

\section{Introduction and results}
\subsection{Introduction}
The study of Hausdorff dimension in profinite groups was initiated by Abercrombie \cite{Ab94} and it has received considerable attention in recent times. For instance, the set of possible Hausdorff dimensions of closed subgroups, the \emph{Hausdorff dimension spectrum}, has been widely studied in pro-$p$ groups (see Section~\ref{subsec:hausdorff} for the definition). 
For example, it is proved in \cite[Theorem~1.1]{barnea:hausdorff} that a $p$-adic analytic pro-$p$ group has finite Hausdorff dimension spectrum, consisting only of rational numbers, with respect to its $p$-power filtration. One of the main open questions about Hausdorff dimension asks whether the converse of \cite[Theorem~1.1]{barnea:hausdorff} holds. This has been recently confirmed for solvable pro-$p$ groups \cite{KTZ}. 

Another natural question is: which groups have complete spectrum (\cite[Problem~5]{barnea:hausdorff})? Some examples are $(t\mathbb{F}_p[\![ t[\!],+)$ and $(1+t\mathbb{F}_p[\![ t[\!],\cdot)$ with respect to their $t$-power filtrations, see \cite[Lemma 4.1]{barnea:hausdorff} and \cite[Lemma 4.2]{barnea:hausdorff}, but these are not finitely generated. In fact, it is not even known whether the Hausdorff dimension spectrum of a finitely generated free pro-$p$ group is complete (see \cite[Theorem~4.10]{shalev:newhorizons} and \cite[Problem~4]{barnea:hausdorff}). 

One of the very few families of examples of profinite groups with complete spectrum is given by automorphism groups of rooted trees with respect to the filtration of the level-stabilisers. In \cite[Theorem~2]{abert:hausdorff} the authors prove that the Hausdorff dimension spectrum of the full automorphism group of a rooted tree is complete, nevertheless the proof relies on probabilistic methods and does not give explicit subgroups with a fixed Hausdorff dimension. Later various authors eventually found explicitly subgroups of automorphisms groups of rooted trees with interesting properties (see \cite{sunic:1} and \cite{Si08}). On the other hand, in the special class of profinite \emph{branch} groups it is possible to construct explicitly subgroups of each given dimension as observed by Klopsch and R\"over \cite[Chapter~8]{klopsch:thesis} and their construction is similar in spirit to the proof of our Theorem~\ref{thm:B} below.

Observe that the automorphism group of a rooted tree can be seen as an infinitely iterated permutational wreath product. The goal of this work is to generalize \cite[Theorem~2]{abert:hausdorff} to \emph{infinitely iterated wreath products in product action}. These groups arise as follows; see
Section~\ref{sec:1} for a more detailed description.  Let
$\mS = (S_k)_{k\in \mathbb{N}}$, with $S_k \le \Sym(\Omega_k)$, be a
sequence of finite transitive permutation groups.  The inverse limit
\[
W^\mathrm{pa}(\mS) = \varprojlim W^\mathrm{pa}_n
\]
of the inverse system
$W^\mathrm{pa}_1 \twoheadleftarrow W^\mathrm{pa}_2 \twoheadleftarrow
\ldots$ of finite iterated wreath products
\[
  W^\mathrm{pa}_n  = S_n \,\cwr\, (S_{n-1} \,\cwr\, ( \cdots \,\cwr\,
  S_0 )) \le
                    \Sym(\wt{\Omega}_n) \quad \text{for $\wt{\Omega}_n =
                      \Omega_n^{\big(\Omega_{n-1}^{\big(\iddots^{\Omega_1}\big)} \big)}$}. 
\]
is called the \emph{infinitely iterated wreath product of type $\mS$
  w.r.t.\ product actions}.

By \cite[Theorem~6.2]{Re12} and
\cite{Va15}, every infinitely iterated
wreath product w.r.t.\ product actions $W^\mathrm{pa}(\mS)$, based on a
sequence $\mS$ of finite non-abelian simple permutation groups, is a
finitely generated hereditarily just infinite profinite group that is
not virtually pro-$p$ for any prime~$p$. In \cite{wilson:largehereditarily,Va15,Va16,KV17}
some embedding, generation and presentation properties of such groups
have been established, but many of their features are not yet fully
understood.

\subsection{Results}
The main result of this paper is that certain infinitely iterated wreath products in product action have complete Hausdorff dimension spectrum with respect to their unique maximal descending chain of open normal subgroups. We start with some notation, let $\mathcal{S}=(S_k\le \Sym(\Omega_k))_{k\in \mathbb{N}}$ be a sequence of finite permutation groups. We say that $\mathcal{S}$ is \emph{good} if there exists a constant $A>0$ and a natural number $M_0$ such that $\norm{S_k}\le \norm{S_{k+1}}^A$ for all $k\ge M_0$.

\begin{thmABC}\label{thm:B}
 Let $\mathcal{S}=(S_k\le \Sym(\Omega_k))_{k\in \mathbb{N}}$ be a good sequence of finite transitive permutation groups and let $G$ be the infinitely iterated wreath product in product action of type $\mS$. Set $N_k=\mathrm{ker}(G\rightarrow \wt{S}_k)$ for $k\in \mathbb{N}$ and $\mathcal{G} = \{N_k\}_{k\in \mathbb{N}}$. Then, for every $\a \in [0,1]$ there is a closed subgroup $H^\a$ of $G$ such that $\dim_{\mathrm{H},\mathcal{G}}(H^\a) = \a$. In particular $\mathrm{Spec}_{\mathrm{H},\mathcal{G}}(G) = [0,1]$.
\end{thmABC}

\begin{rmk*}
 \begin{enumerate}
  \item We do not know if Theorem~\ref{thm:B} still holds for non-good sequences. 
  \item We would like to point out that in the proof of the Theorem~\ref{thm:B} we explicitly construct the subgroup $H^\a$ of the given Hausdorff dimension. In fact, our arguments are purely combinatorial.
  \item In general, the Hausdorff dimension of a profinite group might depend on the chosen filtration. On the other hand, it is easy to see that the filtration $\mathcal{G}$ considered above is the unique maximal filtration of open normal subgroups of $G$.
  \item  We also point out that, in general, the closed subgroups constructed in Theorem~\ref{thm:B} are not finitely generated.
 \end{enumerate}
\end{rmk*}
 
As already mentioned, infinitely iterated wreath products in product action associated to sequences of finite non-abelian transitive permutation groups are hereditarily just infinite. So we readily obtain the following Corollary.

\begin{corABC}
There are hereditarily just infinite profinite groups with complete Hausdorff dimension spectrum.
\end{corABC}

\section{Preliminaries}
\subsection{Hausdorff dimension of profinite groups}\label{subsec:hausdorff}
Let $G$ be a profinite group. A \emph{filtration} of $G$ is a chain $(G_i)_{i\in \mathbb{N}}$ of open subgroups $G_i$ of $G$ such that $\bigcap_{i\in \mathbb{N}} G_i=1$.

\begin{defn}
 Let $G$ be a countably based profinite group and let $H$ be a closed subgroup of $G$. Fix a filtration $\mathcal{G}=(G_n)_{n\in \mathbb{N}}$ of open normal subgroups $G_n$ of $G$. The \emph{Hausdorff dimension} of $H$ (with respect to $\mathcal{G}$) is the real number 
 \[
   \dim_{\mathrm{H},\mathcal{G}}(H) = \liminf_{n\rightarrow \infty} \frac{\log\norm{H G_n:G_n}}{\log\norm{G:G_n}}.
 \]
\end{defn}

In \cite{barnea:hausdorff} it is proved that the previous definition coincides with the usual definition of Hausdorff dimension of a profinite group seen as a metric space with the metric induced by the filtration $\mathcal{G}$. See \cite{barnea:hausdorff} for more details on Hausdorff dimension of pro-$p$ groups.

\begin{defn}
 Let $G$, $H$ and $\mathcal{G}$ be as above. The \emph{spectrum} of $G$ (with respect to $\mathcal{G}$) is the set 
 \[
   \mathrm{Spec}_{\mathrm{H},\mathcal{G}}(G) = \{ \dim_{\mathrm{H},\mathcal{G}}(H) \vvert H\le_c G \}
 \]
\end{defn}

It is clear that $\{0,1\}\subseteq \mathrm{Spec}_{\mathrm{H},\mathcal{G}}(G)\subseteq [0,1]$. We will say that a profinite group $G$ has \emph{complete spectrum} (with respect to $\mathcal{G}$) if $\mathrm{Spec}_{\mathrm{H},\mathcal{G}}(G) = [0,1]$.

\subsection{Infinitely iterated iterated wreath products in product action}\label{sec:1}
This section is devoted to the definition of the family of hereditarily just infinite groups introduced in \cite{Va16}. All the actions considered will be \emph{right} actions. 
A \emph{permutation group} is a subgroup of the symmetric group $\Sym(\Omega)$ on some set $\Omega$. For two permutation groups $A\le \Sym(\Omega)$ and $B\le \Sym(\Lambda)$, we denote by $A\cwr B\le \Sym(\Omega^{\norm{\Lambda}})$ the wreath product of $A$ by $B$ with respect to the product action.\footnote{The product action of the wreath product can also be defined on functions $\Omega^\Lambda$, but identifying a function $f:\Lambda = \{\lambda_1,\ldots,\lambda_n\} \to \Omega$ with the $\norm{\Lambda}$-tuple $(f(\lambda_1),\ldots,f(\lambda_n))$ gives an equivalence of permutation groups.}

\begin{defn}
Let $\mS = (S_k)_{k\in \mathbb{N}}$ be a sequence of finite
  permutation groups $S_k \le \Sym(\Omega_k)$.
Define inductively $\wt{\Omega}_1 = \Omega_1$ and
  $\wt{\Omega}_{n+1} = \Omega_{n+1}^{\, \norm{\wt{\Omega}_{n}}}$ for $n \ge 1$.
  The \emph{$n$th iterated wreath product
    $W^\mathrm{pa}_n \le \Sym(\wt{\Omega}_n)$ of type
    $\mS_n = (S_1, \ldots, S_{n})$ with respect to product actions} is given by
  \begin{align*}
    W^\mathrm{pa}_1 & = W^\mathrm{pa}(\mS_1) = S_1 \le \Sym(\wt{\Omega}_1), \\
    W^\mathrm{pa}_{n+1} & = W^\mathrm{pa} (\mS_{n+1}) = S_{n+1} \,\cwr\,
                      W^\mathrm{pa}_{n} \le  \Sym(\wt{\Omega}_{n+1})
                      \qquad \text{for $n \ge 1$}.  
  \end{align*}
  The \emph{infinitely iterated wreath product of type $\mS$ with respect to
    product actions} is the inverse limit
  $W^\mathrm{pa}(\mS) = \varprojlim W^\mathrm{pa}_n$ of the natural
  inverse system
  $W^\mathrm{pa}_1 \twoheadleftarrow W^\mathrm{pa}_2 \twoheadleftarrow
  \ldots$.
\end{defn}

A profinite group $G$ is said to be \emph{just infinite} if it is infinite and every non-trivial closed normal subgroup is open. By \cite[Theorem~3]{grig:branch}, a just infinite profinite group is either a \emph{branch group} or it is virtually a direct power of a \emph{hereditarily just infinite} profinite group. A just infinite group $G$ is hereditarily just infinite if every open subgroup of $G$ is just infinite. While branch groups received a considerable amount of attention in the past, comparatively little is known about hereditarily just infinite groups. In particular, the only known examples of non-(virtually pro-$p$) hereditarily just infinite groups are the groups defined above and the family of examples described in \cite{lucchini:a2gen}. Both these families of examples are obtained via inverse limits of iterated wreath products. 

\section{Notation and some lemmata}
We start this section by fixing some notation.
\begin{notn} \label{not:filtration}
  For a number $x\in \mathbb{R}$ we will write 
 \[
   \lfloor x \rfloor = \max\left\{ n \in \mathbb{Z}\vvert n\le x\right\} \mbox{\quad and \quad} \{x\} =x - \lfloor x \rfloor.
 \]
 for the integer part and the fractional part of $x$, respectively.
 \end{notn}

\begin{notn}\label{not:ntuple}
 Let $T$ be a group,  $n\in \mathbb{N}$ and $1\le i \le n$ be an index. We will write
\[
  T_i = \left\{ (t_1,\ldots,t_n) \in T^n \vvert  t_j=e  \mbox{ for $j\neq i$}\right\} \le T^n,
\]
that is the $i$-th coordinate subgroup of $T^{n}$.
\end{notn}

 Before the proof of Theorem~\ref{thm:B} we give a few ancillary lemmas.
  The following lemma is straightforward.
 \begin{lem}\label{lem:complement}
  Let $G\le \Sym(\Omega)$ be a permutation group and set $S$ be a subset of $\Omega$. Then $S$ is $G$-invariant if and only if the complement of $S$ in $\Omega$ is $G$-invariant.
 \end{lem}

 The next lemmas are of analytical flavour.
 
 \begin{lem}\label{lem:mn}
  Let $(m_k)_{k\in\mathbb{N}}$ be a sequence of positive integers with $m_k\ge 2$ for every $k$. Let $\wt{m}_1 = m_1$ and $\wt{m}_{k+1}= m_{k+1}^{\wt{m}_k}$ for $k\ge 1$. Then 
  \begin{enumerate}
  \item for every $n\in \mathbb{N}$, $\wt{m}_n \ge n$. In particular, $\lim_{n\rightarrow \infty} \wt{m}_n = \infty$;
   \item  $\displaystyle\lim_{n\rightarrow \infty}  \wt{m}_{n-1} /\wt{m}_n = 0$;
  \item for every positive constant $C$ there exists $M= M(C)\in \mathbb{N}$ such that $C\wt{m}_{n-1} \le \wt{m}_n$ for every $n\ge M$.
  \end{enumerate}
  \end{lem}
  
  \begin{proof}  
  \begin{enumerate}
   \item By induction on $n$. We have $m_1 \ge 2 \ge 1$. Suppose $\wt{m}_{n-1} \ge n-1$ for $n\ge 2$, then $m_n^{\wt{m}_{n-1}} \ge 2^{\wt{m}_{n-1}} \ge 2^{n-1} \ge n$.

   \item Let $x= \wt{m}_{n-1}$, then $0 \le x/m_n^x \le x/2^x$ for every $n\in \mathbb{N}$. Passing to the limits we obtain the claim.
\item Clear from the above. \endproof
  \end{enumerate}
\end{proof}

The previous lemma describes the very fast growth of the function $k\mapsto \wt{m}_k$. The next lemma is a standard result and it can be found in any basic text of Calculus.
  
 \begin{lem}\label{lem:liminf}
  Let $(a_n)_{n\in \mathbb{N}}$ and $(b_n)_{n\in \mathbb{N}}$ two bounded real sequences. Suppose that $(a_n)_{n\in \mathbb{N}}$ converges to $0$, then 
  \[
    \lim_{n\rightarrow \infty} a_n b_n = 0.
  \]
 \end{lem}

 \section{Proof of Theorem~\ref{thm:B}} 
 Set $N_k=\mathrm{ker}(G\rightarrow \wt{S}_k)$ for $k\in \mathbb{N}$, then $\mathcal{G} = (N_k)_{k\in \mathbb{N}}$ be the unique maximal descending chain of open normal subgroups of $G$. Set $\dim_{\mathrm{H}} = \dim_{\mathrm{H},\mathcal{G}}$. Clearly, $\dim_{\mathrm{H}}(\{1\}) = 0$ and $\dim_{\mathrm{H}}(G) = 1$. To prove the theorem it will be sufficient to build subgroups $H^\a_n$ of $W^\mathrm{pa}_{n}$, for $n\in \mathbb{N}$, such that $H_{n+1}^\a $ projects onto $H_n^\a$ and 
   \begin{equation}\label{eq:quant}
   \lim_{n\rightarrow \infty} \frac{\log\norm{H^\a_n}}{\log\norm{G:N_n}} = \a
  \end{equation}
    for every $\a\in (0,1)$. 
  
  For ease of notation, set $m_k = \norm{\Omega_k}$, $\wt{m}_0=1$ and $\wt{m}_k =\norm{\wt{\Omega}_k}$. Observe that the equality $\wt{m}_{k+1}=m_{k+1}^{\wt{m}_k}$ holds for every $k\ge 1$.
  
  Fix $\a \in (0,1)$. We are going to define subgroups $H_n^\alpha$ ``layer by layer'', i.e., we will define subgroups $K_j^\alpha \le S_j^{\wt{m}_{j-1}}$, for $j=2,\ldots,n$, and then we will set $H_n^\alpha = \prod_{j=2}^n K_j^\alpha$. 
   Remembering Notation~\ref{not:ntuple}, define $c_1= m_1$, $o_1 = 1$ and
  \[
    K^\a_2 = \prod_{i=1}^{\left\lfloor \a \cdot m_1\right\rfloor} (S_2)_i \le W^{\mathrm{pa}}_2 \le \Sym(\wt{\Omega}_2).
  \]
  By the definition of the product action, it is clear that $K^\a_2$ has $c_2 = m_2^{m_1- \left\lfloor \a \cdot c_1\right\rfloor \cdot o_1} = m_2^{m_1- \left\lfloor \a \cdot m_1\right\rfloor}$ orbits on $\wt{\Omega}_2$ and each orbit of $K^\a_2$ has the same cardinality $o_2 =  m_2^{\left\lfloor \a \cdot c_1\right\rfloor \cdot o_1}$.
  
  Let $n\ge 2$ and assume that we defined subgroups $K^\a_j $ of $S_j^{\wt{m}_{j-1}}$, for $j=2,\ldots,n$, such that: 
  \begin{enumerate}
   \item[$(a)$] $K^\a_j$ has exactly $c_j$ orbits $\{O_j(1),\ldots,O_j(c_j)\}$ for its action on $\wt{\Omega}_{j}$ and each orbit has cardinality $o_j$;
   \item[$(b)$] $\wt{m}_j = c_j \cdot o_j$ and
   \item[$(c)$]  
   the subset $\mathcal{O}_n = \displaystyle\bigcup_{i=1}^{\left\lfloor \a \cdot c_{n}\right\rfloor} O_{n}(i) \subset \wt{\Omega}_n$ is $K_j^\a$-invariant.
  \end{enumerate}
  Define a new subgroup $K^\a_{n+1}$ of $S_{n+1}^{\wt{m}_n}$ by
  \comm{
  \[
   K^\a_{n+1} = \prod_{i=1}^{\left\lfloor \a \cdot c_n\right\rfloor} S_{n+1}^{\wt{m}_n}(O_n(i))  \le W^{\mathrm{pa}}_{n+1} \le \Sym(\wt{\Omega}_{n+1}).
  \]
  \[
     K^\a_{n+1}  =  \prod_{i=1}^{\left\lfloor \a \cdot c_n\right\rfloor}  \left\{ (s_1,\ldots,s_{\wt{m}_n}) \in S_{n+1}^{\wt{m}_n} \vvert  s_i=e  \mbox{ for $i\notin O_n(i)$} \right\}
  \]}
  \[
    K^\a_{n+1}  =  \prod_{i \in \mathcal{O}_n} (S_{n+1})_i \le W^{\mathrm{pa}}_{n+1} \le \Sym(\wt{\Omega}_{n+1}),
  \]
By definition of product action, the number of orbits of $K^\a_{n+1}$ on $\wt{\Omega}_{n+1}$ corresponds to the number of possible choices for the coordinates of $\Omega_{n+1}^{\wt{m}_n}$ that are not moved by $K^\a_{n+1}$, that is $(m_{n+1})^{\wt{m}_n - \left\lfloor \a \cdot c_n\right\rfloor \cdot o_n}$. Furthermore, the size of an orbit of $K_{n+1}^\a$ on $\wt{\Omega}_{n+1}$ will simply be
  \[
    o_{n+1} = \frac{\wt{m}_{n+1}}{(m_{n+1})^{\wt{m}_n - \left\lfloor \a \cdot c_n\right\rfloor \cdot o_n}} = (m_{n+1})^{\left\lfloor \a \cdot c_n\right\rfloor \cdot o_n}.
  \]
  Therefore $K^\a_{n+1}$ satisfies properties $(a)$ and $(b)$, we will prove that $K^\a_{n+1}$ and $\mathcal{O}_{n+1}$ also satisfy property $(c)$. It is clear that $\mathcal{O}_{n+1}$ is $K_{n+1}^\a$-invariant.
  Set $C = \wt{\Omega}_n \setminus \mathcal{O}_n$. 
  By definition of $K_{n+1}^\a$, for any orbit $O$ of $K_{n+1}^\a$ on $\wt{\Omega}_{n+1}$ and for any $c\in C$ there exist $f_c\in \Omega_{n+1}$ such that   
  \[
    O = \left\{(x_1,\ldots,x_{\wt{m}_n}) \in \wt{\Omega}_{n+1} \vvert x_c = f_c \mbox{ for } c\in C \right\}.
  \]
  It follows that there is a bijection between orbits of $K_{n+1}^\a$ on $\wt{\Omega}_{n+1}$ and the set $ \Omega_{n+1}^C$. By Lemma~\ref{lem:complement}, property $(c)$ yields that $C$ is $K_j^\a$-invariant for every $j=2,\ldots,n-1$ and this implies that $\mathcal{O}_{n+1}$ is $K_j^\a$-invariant. Therefore, $K_{n+1}^\a$ satisfies $(c)$.
  
 Set $H_2^\a = K_2^\a$ and $H_{n+1}^\a = H_n^\a \cdot K^\a_{n+1}$ for $n\ge 2$. By property $(c)$, it follows readily that $H_n^\a$ is a subgroup of $W_n^{\mathrm{pa}}$. Also, by construction, $H^\a_{n+1}$ projects onto $H^\a_n$.  Set $H^\a = \varprojlim H_n^\a$.
  
 The closed subgroup $H^\a$ of $G$ is our candidate to have Hausdorff dimension $\a$ in $G$ (with respect to $\mathcal{G}$). In the rest of the proof we will prove that this is indeed the case. 
 An algebraic manipulation yields that
  \[
   \log \norm{H_n^\a} = \log \left(\prod_{k=2}^{n} \norm{S_k}^{\left\lfloor \a \cdot c_{k-1}\right\rfloor \cdot o_{k-1}} \right) = \sum_{k=2}^n \left\lfloor \a \cdot c_{k-1}\right\rfloor \cdot o_{k-1}  \log\norm{S_k}
  \]
  and, defining $\wt{m}_0=1$,
  \[
   \log \norm{G:N_n} = \log \left(\prod_{k=1}^{n} \norm{S_k}^{\wt{m}_{k-1}} \right) = \sum_{k=1}^n \wt{m}_{k-1} \log\norm{S_{k}}.
  \]
  We have to carefully study the asymptotics of the previous sequences. 
  First we determine the asymptotic behaviour of the sequence $(\left\lfloor\a c_n \right\rfloor o_n/\wt{m}_n)_{n\ge 1}$. Observe that 
  \[
   \frac{\left\lfloor\a \cdot c_n \right\rfloor \cdot o_n}{\wt{m}_n} = \frac{\a \cdot c_n \cdot o_n - \left\{\a \cdot c_n \right\} \cdot o_n }{\wt{m}_n} = \a - \frac{\left\{\a \cdot c_n \right\} \cdot o_n }{\wt{m}_n}
  \]
  and, by definition of $c_n$ and Lemma~\ref{lem:mn}, $c_n\ge \wt{m}_n^{(1-\a)} \ge n^{(1-\a)}$ which tends to infinity as $n$ does. Moreover, remember that $o_n/\wt{m}_n = 1/c_n$. Therefore
 \begin{equation}\label{eq:limit2}
   \lim_{n\rightarrow \infty} \frac{\left\lfloor\a \cdot c_n \right\rfloor \cdot o_n}{\wt{m}_n} = \a.
 \end{equation}

  We are going to show next that 
  \[
  \lim_{n\to \infty} \frac{\sum_{k=2}^n \left\lfloor \a \cdot c_{k-1}\right\rfloor \cdot o_{k-1}  \log\norm{S_k}}{ \left\lfloor \a \cdot c_{n-1}\right\rfloor \cdot o_{n-1}  \log\norm{S_n}} = 1.
  \]

  For the sake of brevity, set $$a_k = \left(\a \wt{m}_{k-1} \log\norm{S_{k}}\right)_{k\ge 2} \text{\quad and \quad} b_k = \left(\left\lfloor \a  c_{k-1}\right\rfloor o_{k-1} \log\norm{S_k}\right)_{k\ge 2},$$ then it is clear that $b_k\le a_k$, for every $k\ge 2$. By \eqref{eq:limit2}, the sequence $(b_n/a_n)_{n\ge 2}$ tends to $1$ as $n$ tends to infinity. 
    
  In the next series of claims we will show that $(\sum_{k=2}^n a_k)_{n\ge 2}$ ``behaves asymptotically'' like $(a_n)_{n\ge 2}$. Recall that we are assuming that there are $M_0\in \mathbb{N}$ and $A>0$ such that, for all $k\ge M_0$, $ \log\norm{S_k}\le A \log\norm{S_{k+1}}$.
  
 \begin{claim}
 For every real constant $C>0$ there exists $M(C)\in \mathbb{N}$ such that  
 \begin{equation}\label{eq:ak}
 C a_{k-1} \le a_k \text{\quad for every $k\ge M(C)$}.
 \end{equation}
  \begin{proof}
  By Lemma~\ref{lem:mn}, there exists a natural number $L =L(C A)$ such that $C A \wt{m}_{k-2} \le \wt{m}_{k-1}$ for every $k\ge L$. Thus, $$Ca_{k-1} \le C A \a \wt{m}_{k-2} \log \norm{S_{k}}  \le  a_k,$$ for every $k\ge M(C) = \max\{L,M_0+1\}$. \qedhere
 \end{proof}
 \end{claim}
 
 \begin{claim}
 For all $n\ge M(2)$,
 \begin{equation}\label{eq:sumlessthan2}
  \sum_{k=M(2)}^{n} a_k \le 2 a_{n}.
 \end{equation}
 \begin{proof}
 This will be proved by induction on $n$. It is clear that $a_{M(2)}\le 2a_{M(2)}$. Suppose by inductive hypothesis that $\sum_{k=M(2)}^{n-1} a_k \le 2 a_{n-1}$, then \eqref{eq:ak} yields $$\sum_{k=M(2)}^n a_k \le 2 a_{n-1} + a_n\le 2a_n,$$ for $n\ge M(2)$. \qedhere
 \end{proof}
 \end{claim}
 
 \begin{claim}
 For all $n\ge \wh{M}=\max\{M(2)+1,M(M(2))\}$,
 \begin{equation}\label{eq:nice}
 \sum_{k=2}^{n} a_k \le 3 a_{n} .
 \end{equation}
\begin{proof}
  By \eqref{eq:ak}, $M(2)a_{M(2)}\le M(2) a_{n-1}\le a_n$ for all $n\ge \wh{M}$. Using \eqref{eq:sumlessthan2}, we conclude that $$\sum_{k=2}^{n} a_k \le M(2)a_{M(2)} + 2 a_n \le 3a_n,$$ 
  for every $n\ge \wh{M}$.
\end{proof}
\end{claim}

Finally, we use inequality \eqref{eq:nice} to see that
  \begin{equation}\label{eq:nice2}
  0\le \frac{1}{a_n} \cdot \sum_{k=2}^{n-1} a_k \le \frac{3a_{n-1}}{a_n} 
  \end{equation}
  for every $n$ large enough. 
 Since the sequence $(\log\norm{S_{n-1}}/\log\norm{S_{n}})_{n\in \mathbb{N}}$ is positive and bounded above by the constant $A$, it follows from Lemma~\ref{lem:mn} and Lemma~\ref{lem:liminf} that
  \begin{equation}\label{eq:nice3}
   \lim_{n\rightarrow \infty} \frac{a_{n-1}}{a_n} = \lim_{n\rightarrow \infty} \frac{\wt{m}_{n-2}}{\wt{m}_{n-1}} \cdot \frac{\log\norm{S_{n-1}}}{\log\norm{S_n}} =  0
  \end{equation}
 and from \eqref{eq:nice2} we deduce that, 
\begin{equation}\label{eq:limit11}
   \lim_{n\rightarrow \infty} \sum_{k=2}^n \frac{a_{k}}{a_{n}} =\lim_{n\rightarrow \infty} \sum_{k=2}^n \frac{\wt{m}_{k-1} \log \norm{S_{k}}}{\wt{m}_{n-1} \log \norm{S_{n}}} = 1.
  \end{equation}
  We will use equation \eqref{eq:limit11} at the end of the proof. 
  
  We now look at the asymptotics of the sequence $(b_k)_{k\ge 2}$. Again using \eqref{eq:nice2} and the fact that $b_k\le a_k$ for all $k\ge 2$, we have
  \[
    0\le \left(\sum_{k=2}^{n-1} \frac{b_k}{b_n} \right) \cdot \frac{b_n}{a_n} \le \sum_{k=2}^{n-1} \frac{a_k}{a_n} \le \frac{3a_{n-1}}{a_n} 
  \]
and, since $(b_n / a_n)_{n\ge 2}$ tends to $1$ for $n$ that tends to infinity, it follows that
  \begin{equation}\label{eq:limit1}
   \lim_{n\rightarrow \infty} \sum_{k=2}^n \frac{b_k}{b_n} = 1.
  \end{equation}
  
 Let us now summarize everything in the proof of \eqref{eq:quant}. As previously observed, we can write the Hausdorff dimension of $H^\a_n$ as:
  \begin{equation}\label{eq:last}
  \dim_{\mathrm{H}}(H^{\a}) = \liminf_{n\rightarrow \infty} \frac{\log\norm{H_n^\a}}{\log\norm{G:N_n}} = \liminf_{n\rightarrow \infty} \frac{\sum\limits_{k=2}^{n} b_k}{\sum\limits_{k=1}^{n} \wt{m}_{k-1}  \log\norm{S_k}}.
  \end{equation}
  Define $a_1=\a \wt{m}_0 \log\norm{S_1}$ and observe that $\lim_{n\rightarrow \infty} a_1/a_n=0$. Collecting the highest terms in the top and in the bottom of the fraction in \eqref{eq:last}, we get  
  \begin{equation*}
    \dim_{\mathrm{H}}(H^{\a}) =  \liminf_{n\rightarrow \infty}\frac{ \sum\limits_{k=2}^n \frac{b_k}{b_n}}{  \sum\limits_{k=1}^{n} \frac{a_k}{a_n}} \cdot   
  \frac{\left\lfloor \a \cdot c_{n-1}\right\rfloor \cdot o_{n-1}}{\wt{m}_{n-1}}. 
 \end{equation*}
 By \eqref{eq:limit1} and \eqref{eq:limit11}, the \emph{limit} of the first factor of the above product is $1$ and, by \eqref{eq:limit2}, the \emph{limit} of  second factor is $\a$. Therefore $\dim_{\mathrm{H}}(H^\a) = \a$, as claimed. This concludes the proof. \qed

 \begin{rmk*}
  We would like to point out that in the proof of Theorem~\ref{thm:B} all the limits considered are actual limits and not inferior limits.
 \end{rmk*}

 \section*{Acknowledgements}
We thank the anonymous referee for their comments on the exposition of this work. The second author acknowledges the support of Royal Holloway College for the Ph.D.\ grant under which part of this work was carried out.

\bibliographystyle{plain}
\def\cprime{$'$}
\end{document}